\documentclass[11 pt]{article}
\usepackage{graphicx, multirow}
\usepackage{amssymb,amsmath,amsthm, latexsym}

\newtheorem{theorem}{Theorem}
\newtheorem{lemma}{Lemma}

\begin{document}
\setlength{\oddsidemargin}{-.5in}
\setlength{\evensidemargin}{-.5in}

\title{Classification of Book Representations of $K_6$}
\author{Dana Rowland}
\date{March 13, 2017}

\maketitle

\abstract{A book representation of a graph is a particular way of embedding a graph in three dimensional space so that the vertices lie on a circle and the edges are chords on disjoint topological disks.  We describe a set of operations on book representations that preserves ambient isotopy, and apply these operations to $K_6$, the complete graph with six vertices. We prove there are exactly 59 distinct book representations for $K_6$, and we identify the number and type of knotted and linked cycles in each representation.  We show that book representations of $K_6$ contain between one and seven links, and up to nine knotted cycles.  Furthermore, all links and cycles in a book representation of $K_6$ have crossing number at most four.}

\section{Introduction}

Take any six points in three-dimensional space and connect each pair of points with distinct simple curves.  The result is a spatial embedding of the complete graph $K_6$.  The points are the vertices of the graph and the curves joining the points are the edges.   Cycles in a spatial embedding form (possibly trivial) knots in $\mathbb{R}^3$ and disjoint pairs of 3-cycles form links.  A well-known result of Conway and Gordon establishes that every spatial embedding of $K_6$ must contain at least one non-trivial link \cite{CG}.

One way to connect the six vertices is to use line segments.  In this case, the embedding is called a linear embedding.  In a linear embedding, each cycle is a polygonal representation of a knot.  The polygonal index or stick number of a knot, $s(K)$, is defined to be the least number of line segments that can be used to form a knot.  Since the trefoil knot is the only non-trivial knot with $s(K)\leq 6$, all cycles of a linear embedding of $K_6$ must either be trivial or trefoil knots.  In particular, Huh and Jeon proved that a linear embedding of $K_6$ contains either exactly one Hopf link and no knotted cycles, or exactly three Hopf links and one trefoil knot \cite{hj}.

Note that in a projection of a linear embedding, no edge can intersect itself and no edge intersects an adjacent edge.  We consider a more general collection of spatial embeddings that preserve this property, known as book representations.  A book representation is a spatial embedding of a graph $G$ so that the vertices are on a line (the ``spine'' of the book) and edges are semicircles on disjoint half-planes (the ``pages'') meeting at the spine.

An equivalent description is obtained by bending the spine to form a circle; then the conditions for a book representation become:
\begin{enumerate}
\item The vertices of $G$ lie on a circle $C$ in a plane.
\item The edges of $G$ are distributed among disjoint \emph{sheets}, topological disks with boundary $C$.
\item The projection of edges of $G$ to the plane containing $C$ are chords of $C$.

\end{enumerate}
Book representations of graphs have applications to layout design of multi-layer circuit boards; a chip designer needs to avoid wires crossing within the same layer, just as edges cannot cross within the same sheet.  See for example \cite{chung}.

Figures \ref{4sheets} and \ref{4sheetK6} provide an example of how separate sheets are combined to form a book representation of $K_6$ with 4 sheets.
\begin{figure}[h]
\centering
\includegraphics[width = 0.9\linewidth]{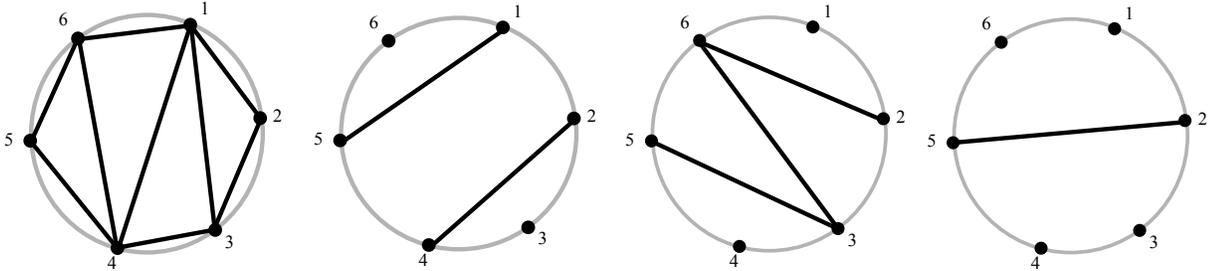}
\caption{Each edge of $K_6$ appears in exactly one of the four disks.}
\label{4sheets}
\end{figure}

\begin{figure}[h]
  \begin{minipage}[b]{0.50\linewidth}
    \centering
    \hfill
    \includegraphics[width = 0.4\linewidth]{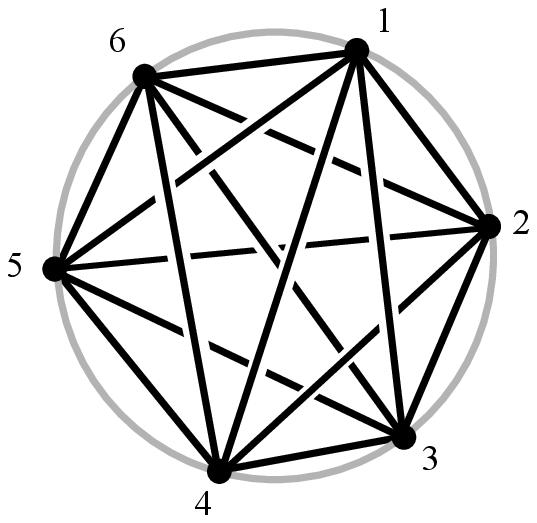}
    \par\vspace{0pt}
  \end{minipage}%
  \begin{minipage}[b]{0.50\linewidth}
    \centering%
    \begin{tabular}{|c|c|}
        \hline
        Sheet 1 & 13,14,46\tabularnewline
        \hline
        Sheet 2 & 15,24\tabularnewline
        \hline
        Sheet 3 & 26,35,36\tabularnewline
        \hline
        Sheet 4 & 25\tabularnewline
        \hline
        \end{tabular}
    \par\vspace{15pt}
  \end{minipage}
\caption{A book representation of $K_6$ is obtained by stacking the disks and identifying their boundaries.}
\label{4sheetK6}
\end{figure}

A cycle in a book representation is an arc presentation of a knot.  See \cite{cromwell} for an overview of the arc presentation.  The arc index $a(K)$ of a knot is the minimum number of edges in any arc presentation of the knot; $a(3_1)=5$, $a(4_1)=6$, and $a(K) > 6$ for all other non-trivial knots.  Since a cycle in $K_6$ contains at most 6 edges, the arc index implies each cycle in a book representation of $K_6$ is either trivial, a trefoil, or a figure-eight knot.  Furthermore, each pair of disjoint 3-cycles in a book representation of $K_6$ forms either a trivial link, a Hopf link ($2_1^2$), or a Solomon link ($4_1^2$).

The minimal sheet book representation of $K_6$ contains exactly one Hopf link and no knotted cycles.  We prove that a book representation of $K_6$ can contain up to seven non-trivial two component links and up to nine non-trivial knotted cycles.  A list of all 59 book representations and the knotted cycles and links that appear in each appears in section \ref{appendix}.

The outline of this paper is as follows.  In Section \ref{sec:equiv}, we introduce a set of moves from one book representation to another that will not change the ambient isotopy type of an embedding.  In Section \ref{sec:3sheets} we review Otsuki's canonical book representation of $K_6$ and show that every representation with 3 sheets is ambient isotopic to this one.  In Sections \ref{sec:4sheets} through \ref{sec:moresheets}, we describe how to find an additional 58 distinct book representations with between 4 and 9 sheets, and we prove they are distinct. The primary technique involves enumerating all ways to partition interior edges to a given number of sheets, determining equivalences using the moves from \ref{sec:equiv}, and then showing the remaining representations are distinct by considering the set of knotted and linked cycles.

\section{Equivalence of Book Representations}
\label{sec:equiv}

Recall that $K_n$, the complete graph on $n$ vertices, is a graph where each vertex has exactly 1 edge connecting it to each of the other $n$-1 vertices.  To describe a book representation of $K_n$, we label the vertices from $1$ to $n$ traveling clockwise around the boundary circle, and let edge $ab$ denote the edge connecting vertex $a$ to vertex $b$, where $a<b$.  We define the \emph{length} of edge $ab$ to be the distance between $a$ and $b$ along the boundary circle.  The formula is given by

\[ \mathrm{length}(ab) = \mathrm{min}[(b-a), (a-b) \mathrm{mod}\ n].    \]

We call edges of length 1 \emph{exterior} edges.  Note that these edges never intersect any other edges when projected to the plane containing the boundary circle.  Edges of length 2 or more are called \emph{interior} edges. A book representation is defined by assigning a sheet to each interior edge in such a way that edges in the same sheet do not intersect in the projection plane.  The sheet number of a book representation is the minimum number of sheets required, where the minimum is taken over all ambient isotopic book representations. For $K_6$, the minimum sheet number is 3. For $K_n$, the minimum possible sheet number is given by the formula $\lceil {n/2} \rceil$ \cite{bernhart,kobayashi}. Otsuki's canonical book representation of $K_n$ attains this minimum sheet number \cite{otsuki}. The canonical book representation of $K_6$ is given in Section \ref{sec:3sheets}.

On the other extreme, the maximal possible sheet number is equal to the number of interior edges.  In total, $K_n$ has $\dfrac{n(n-1)}{2}$ edges.  The number of interior edges is then given by the formula

	\[\frac{n(n-1)}{2} - n = \frac{n(n-3)}{2}\]

Since the complete graph $K_6$ has 15 edges total--6 exterior edges and 9 interior edges--the maximal possible sheet number 9.

A particular sheet of a book representation of $K_6$ can contain at most three interior edges.  Table \ref{tbl:edges} shows the possible pairs and triples of interior edge assignments that can occur within a sheet.  Note that in $K_6$, an edge of length three projects to a diameter of the boundary circle.  Every pair of edges of length three intersect; therefore, each sheet can contain at most one edge of length three.  We will refer to these edges--14, 25 and 36--as \emph{long} edges, and the remaining interior edges will be \emph{short} edges.
\begin{table}
\centering
  \begin{tabular}{|c|c|c|c|c|} \hline
 \multicolumn{3}{|c|}{\textbf{Two edges}} & \multicolumn{2}{|c|}{\textbf{Three edges}}\\ \hline
  13, 14   &  14, 46 & 25, 26   &13, 14, 15 & 15, 24, 25\\
  13, 15   &  15, 24 & 25, 35   &13, 14, 46 & 15, 25, 35\\
  13, 35   &  15, 25 & 26, 35   &13, 15, 35 & 24, 25, 26\\
  13, 36   &  15, 35 & 26, 36   &13, 35, 36 & 24, 26, 46\\
  13, 46   &  24, 25 & 26, 46   &13, 36, 46 & 25, 26, 35\\
  14, 15   &  24, 26 & 35, 36   &14, 15, 24 & 26, 35, 36\\
  14, 24   &  24, 46 & 36, 46   &14, 24, 46 & 26, 36, 46\\

    \hline
    \end{tabular}
  \caption{Possible two or three edge sheets in a book representation of $K_6$}\label{sheets}
  \label{tbl:edges}
\end{table}

Any way of partitioning the interior edges of $K_6$ into sheets of non-intersecting edges will produce a book representation; however, many of these partitions will describe the same spatial embedding up to ambient isotopy.  To classify all book representations, we must determine when two edge assignments represent equivalent book representations.  In \cite{cromwell}, Cromwell establishes a set of operations on arc presentations of knots so that any two arc presentations of the same knot are related by a finite sequence of operations. The following theorem describes an analogous set of operations which can be performed on book representations without changing the ambient isotopy type of the spatial graph.

\begin{theorem} \label{thm:moves} The following moves result in ambient isotopies for book representations of $K_n$ in $S^3$:
\begin{enumerate}
\item Rotating vertices:  Increase or decrease all vertex labels by one, mod $n$.  \label{rotate}
\item Shifting sheets:  Let $s$ be the number of sheets.  Increase or decrease all sheet numbers by one, mod $s$.  \label{shift}
\item Edge move:  If an edge does not intersect any of the edges in an adjacent sheet, then move the edge into that sheet.  \label{bump}
\item Vertex exchange:  If $v$ and $w$ are consecutively labeled vertices on the boundary circle, and if the sheet number of each edge incident to $v$ is higher than the sheet number of any edge incident to $w$, then exchange the positions of vertices $v$ and $w$ and then relabel.  \label{vexchange}
\item Sheet insertion/deletion:  Add or remove an empty sheet between any two sheets.  \label{addsheet}
\item Double reflection:  Compose any pair of reflections through a diameter of the boundary circle or through the plane containing the boundary circle. \label{mirror2}
\end{enumerate}
\end{theorem}
\begin{proof}
Consider a projection of a book representation to the plane containing the boundary circle $C$.  Since moves \ref{bump} and \ref{addsheet} do not change the projection, they are ambient isotopies.  Increasing or decreasing vertex labels is equivalent to rotating the graph $2\pi/n$ radians around the center of the boundary circle.  The composition of two reflections in Move \ref{mirror2} also results in a rotation in $S^3$, so Moves \ref{rotate} and \ref{mirror2} are ambient isotopies.   Move \ref{shift}, shifting sheets, requires all edges in the top sheet to be moved to the bottom sheet (or vice versa), and an example of vertex exchange, move \ref{vexchange}, is shown in Figure \ref{fig:vexchange}.  Both of these can be achieved using the generalized Reidemeister moves for spatial graphs shown in Figure \ref{Reid}; see \cite{kauffman} for a discussion of generalized Reidemeister moves.
\end{proof}

\begin{figure}
\label{Reid}
\centering
\includegraphics{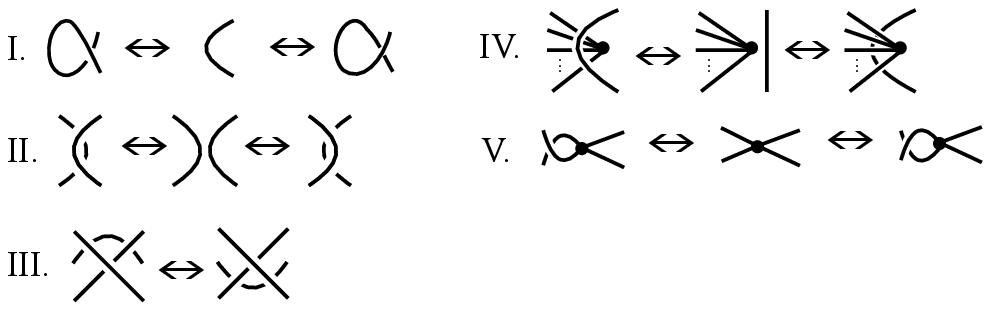}
\caption{The generalized Reidemeister moves}
\end{figure}

\begin{figure}\label{fig:vexchange}
\begin{minipage}[b]{0.50\linewidth}
    \centering
\includegraphics[width = \linewidth]{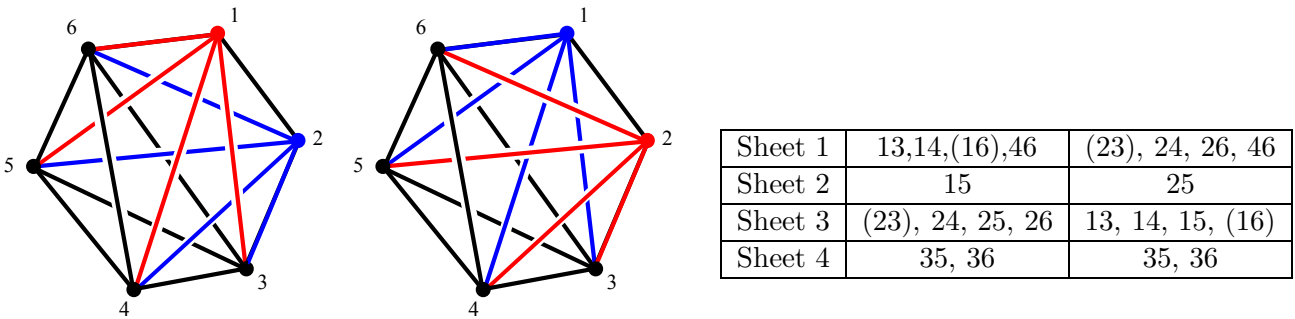}
\par\vspace{0pt}
  \end{minipage}%
  \begin{minipage}[b]{0.50\linewidth}
    \centering%
    \begin{tabular}{|c|c|c|}
        \hline
        Sheet 1 & 13,14,(16),46 & (23), 24, 26, 46\tabularnewline
        \hline
        Sheet 2 & 15 & 25\tabularnewline
        \hline
        Sheet 3 & (23), 24, 25, 26 & 13, 14, 15, (16)\tabularnewline
        \hline
        Sheet 4 & 35, 36 & 35, 36 \tabularnewline
        \hline
        \end{tabular}
    \par\vspace{15pt}
  \end{minipage}
\caption{Vertex exchange of adjacent vertices is possible through a sequence of Reidemeister IV moves--in this example, the blue vertex passes beneath all the edges incident to the red vertex.}
\end{figure}

In later sections, we use the equivalences in Theorem \ref{thm:moves} to show that every book representation of $K_6$ is equivalent to one of the representations shown in Section \ref{appendix}.

\section{$K_6$ with Minimum Sheet Number}
\label{sec:3sheets}
In this section, we prove there is a unique 3-sheet book representation of $K_6$.

\begin{theorem}
Every 3-sheet book representation of $K_6$ is equivalent to the following:

\[\begin{tabular}{|c|c|}
\hline
Sheet 1 & 13, 14, 46\tabularnewline
\hline
Sheet 2 & 15, 24, 25\tabularnewline
\hline
Sheet 3 & 26, 35, 36\tabularnewline
\hline

\end{tabular}\]
\end{theorem}

\begin{proof}
In a minimal book representation of $K_6$, each sheet must contain one long edge (of length three) and two short edges (of length two).  By shifting the sheets, we can assume that edge $14$ is in the top sheet.   Refer the three-edge sheets which contain edge 14 shown in Table \ref{tbl:edges}.  There are two cases to consider:  either edge $13$ or edge $24$ is also in sheet 1.  If edge $13$ is in sheet 1, then edge $26$ must be in the same sheet as edge $36$ since, as shown in Table \ref{tbl:edges}, all the potential sheets with three edges that contain 36 and not 13 contain 26. Similarly, edge $15$ must be in the same sheet as edge $25$.  There is a unique way to complete the sheets with the remaining edges to obtain three non-intersecting edges per sheet.  Similarly, if edge $24$ is in sheet 1, then edge $35$ is in the same sheet as edge $25$ and edge $13$ must be in the same sheet as edge $36$.

That means the only four possible options are:
\[
\begin{tabular}{|c|c|c|c|c|}
\hline
& \textbf{Option 1} & \textbf{Option 2} & \textbf{Option 3} & \textbf{Option 4}\tabularnewline
\hline
\hline
Sheet 1 & 13, 14, 46 &  13, 14, 46 &  14, 15, 24 &  14, 15, 24\tabularnewline
\hline
Sheet 2 &  15, 24, 25 &  26, 35, 36 &  13, 36, 46 &  25, 26, 35\tabularnewline
\hline
Sheet 3 &  26, 35, 36 &  15, 24, 25 &  25, 26, 35 &  13, 36, 46\tabularnewline
\hline
\end{tabular}
\]
Options 1 and 2 are Otsuki's left and right canonical book representations \cite{otsuki}. Note that Option 3 is a mirror image of Option 1, obtained by reflecting across the edge $14$, so that vertices $2$ and $6$ and vertices $3$ and $5$ are exchanged.  Since Option 2 is also a mirror image of Option 1 (obtained by reversing the order of sheets) we know that Options 2 and 3 must be equivalent.  Similarly, Options 1 and 4 are equivalent.

Finally, we show that Option 1 and Option 2 are also equivalent using
the following steps. Starting with option 1:
\begin{enumerate}
\item Insert an empty sheet 4.  (Move \ref{addsheet})
\item Move edges 13, 35 and 36 to sheet 4. (Move \ref{bump})
\item Note that the sheet number of all vertices adjacent to vertex 2 is higher than the sheet number of all vertices adjacent to vertex 3.  Swap vertex 2 and vertex 3.  (Move \ref{vexchange})
\item Move edge 24 to sheet 1.  (Move \ref{bump}).
\item Swap vertex 3 and vertex 4.  (Move \ref{vexchange})
\item Swap vertex 2 and vertex 3.  (Move \ref{vexchange})
\item Move edges 13 and 14 into sheet 3, move edges 24 and 25 into sheet 2, and move edge 26 to sheet 4. (Move \ref{bump})
\item Delete sheet 1. (Move \ref{addsheet})
\item Shift sheets, so \{13, 14,46\} becomes the top sheet.  (Move \ref{shift})
\end{enumerate}
These moves are shown in the table below.  Note that exterior edges can always be freely moved to any sheet; the exterior edges are only included below in parentheses when they are adjacent to a vertex used in a vertex exchange.
\[\begin{tabular}{|c|c|c|c|}
\hline
  Add Sheet & Move edges & Vertex exchange & Move edge \\
\hline
13, 14, 46 &  14, 46  & 14, 46 & 14, 24, (45), 46 \\
15, 24, 25 &  (12), 15, 24, 25 & 13, 15, (34), 35 & 13, 15, 35 \\
26, 35, 36 &  26            & 36 & (23), 36 \\
            & 13, (34), 35, 36 & (12), 24, 25, 26 & 25, 26 \\
\hline
Vertex exchange & Vertex exchange & Move edges & Delete and Shift sheets\tabularnewline
\hline
13, (23), (34), 35, 36 & (12), 24, 25, 26 & &13, 14, 46\tabularnewline
14, 15, (45) & 14, 15 & 15, 24, 25 & 26, 35, 36\tabularnewline
24, 46 & (34),46 & 13, 14, 46 & 15, 24, 25\tabularnewline
(12), 25, 26 & 13, 35, 36  & 26, 35, 36 & \tabularnewline
\hline
\end{tabular}\]
This shows that there is only one minimal sheet book representation of $K_6$, which is achiral (equivalent to its mirror image).
\end{proof}
Note that this embedding of $K_6$ contains no knotted cycles, and exactly one link--the cycles 135 and 246 form a Hopf link.

\section{Generating $4$-sheet Book Representations of $K_6$}
\label{sec:4sheets}
Next, we describe our method for generating book representations for complete graphs with $6$ vertices on $4$ sheets.  After generating all possible book representations, we then can use the methods from Section \ref{sec:equiv} to determine which of the representations are equivalent.

We begin by considering the possible partitions of edges to the four sheets.

\begin{lemma}
\label{lem:shtdist}
Every 4-sheet book representation of $K_6$ can be described using one sheet containing three edges, and three sheets containing exactly two edges.
\end{lemma}
\begin{proof}
Each sheet must contain between one and three edges, and nine edges total must be assigned to the four sheets.  It follows that at least one sheet has exactly three edges.  By shifting sheets, we can ensure that sheet 1 contains three edges, and that a sheet with three edges is not on the bottom.  This leads to one of the following configurations for distributing interior edges:
  \[ (3,2,2,2), (3,3,2,1), (3,3,1,2), (3,2,3,1)  \]

We know any edge in a projection of a book representation of $K_6$ intersects at most four other edges.  Suppose that a configuration has a sheet with a single edge.  We see above that at least five edges are contained in adjacent sheets, so it is always possible to move a non-intersecting edge into the sheet with a single edge.  Thus, by moving edges if necessary, we can assume that the distribution of interior edges is either $(3,2,2,2)$ or $(3,3,2,1)$.

Now consider the edge distribution $(3,3,2,1)$.  If an edge in sheet 4 does not intersect all of the edges in sheet 1, then an edge may be moved from sheet 1 to obtain a $(3,2,2,2)$ configuration.  Suppose the edge in sheet 4 does intersect all the edges in sheet 1.  Since it must also intersect an edge in sheet 3, it must be a long (length 3) edge. By rotating vertices, we can assume that edge $25$ is in sheet 4, which implies that edges $13, 14, 36$ and $46$ are in sheets 1 and 3.  Table \ref{tbl:edges} lists all possible ways to assign three edges to a sheet.  Note that it is impossible to assign three of the unused edges to sheet 2.  Therefore, every 4-sheet representation can be represented with three edges in sheet 1 and two edges in each of the remaining sheets.
\end{proof}

Next, we prove that we only need to consider two possibilities for the top sheet.
\begin{lemma}
\label{lem:topsheet}
Every 4-sheet book representation of $K_6$ can be expressed so that the set of edges in sheet 1 is either $\{13, 14, 46\}$ or $\{13, 14, 15\}$.
\end{lemma}
\begin{proof}
By Lemma \ref{lem:shtdist}, we can assume the top sheet contains three edges, and all other sheets contain two edges.  If sheet 1 contained all short edges (of length two), then sheet 2 would contain one long edge and one short edge.  This pair of edges cannot intersect all three of the short edges in sheet 1, so some edge from sheet 1 could be moved into sheet 2 without causing an intersection. Therefore, we can assume without loss of generality that one of the long edges is in a sheet with three edges.  By shifting sheets we can ensure that sheet 1 has three edges, and by rotating vertices we can ensure that one of them is the edge 14.  That leaves four possibilities for sheet 1; it could contain edges $\{13, 14, 15\}$, $\{13, 14, 46\}$, $\{14, 15, 24\}$, or $\{14, 24, 46\}$.  If a representation has $\{14, 15, 24\}$ as its first sheet, perform a double reflection:  reflect through the plane containing the boundary circle to reverse the order of the sheets, then reflect across edge 14 to obtain the sheet $\{13, 14, 46\}$.  Finally, shift sheets so this sheet returns to the top.  If a representation has $\{14, 24, 46\}$ as its top sheet, rotate the vertices by adding 3 to each label to obtain $\{13, 14, 15\}$.  Thus, every representation is equivalent to one that has either $\{13, 14, 46\}$ or $\{13, 14, 15\}$ on top, as claimed.
\end{proof}

By using Lemmas \ref{lem:shtdist} and \ref{lem:topsheet} together with the equivalences established in Section \ref{sec:equiv}, we prove the following theorem:
\begin{theorem}
Every 4-sheet book representation of $K_6$ is equivalent to either
\[\begin{tabular}{|c|c|}
\hline
Sheet 1 & 13, 14, 46\tabularnewline
\hline
Sheet 2 & 15, 24\tabularnewline
\hline
Sheet 3 & 26, 36\tabularnewline
\hline
Sheet 4 & 25, 35\tabularnewline
\hline
\end{tabular}
\]
or its mirror image.
\end{theorem}

\begin{proof}
By Lemma \ref{lem:topsheet}, we only need to consider two choices for the first sheet.  If the top sheet contains edges 13, 14, and 46, then there are four possible ways to partition the remaining six edges into three disjoint pairs, each of which can be ordered in $3!=6$ ways, resulting in 24 potential book representations.  It is straightforward to verify, using the moves in Section \ref{sec:equiv}, that all of these representations can either be reduced to three sheets, or are equivalent to one of the following:
\[\begin{tabular}{|c|c|c|}
\hline
Sheet 1 & 13, 14, 46 & 13, 14, 46\tabularnewline
\hline
Sheet 2 & 15, 24 & 25, 35\tabularnewline
\hline
Sheet 3 & 26, 36 & 26, 36\tabularnewline
\hline
Sheet 4 & 25, 35 & 15, 24\tabularnewline
\hline
\end{tabular}
\]
If the top sheet contains edges 13, 14, and 15, then there are five possible ways to partition the remaining six edges into three non-intersecting pairs, each of which can be ordered in $3!=6$ ways, resulting in 30 additional possible book representations.  By moving edges and rotating vertices, one can show that each book representation in this case is equivalent to a book representation with top sheet \{13, 14, 46\}.
Thus, two book representations shown above are the only two distinct representations of $K_6$ with sheet number 4.  These representations are mirror images of each other.

Note that cycle 136425 is a trefoil knot.  This can be clearly seen by redrawing one of the projections using the minimum number of crossings, as shown in Figure \ref{fig:trefoil}.  Since 136426 is the only cycle containing all three crossings, all other cycles must be trivial.  Furthermore, because the left-handed and right-handed trefoil are distinct knots, these two representations are distinct from each other and from the three-sheet book representation of $K_6$ which contains no knotted cycles.
\end{proof}

In the book representation given in the theorem, cycle 136425 is a right-handed trefoil and all other cycles are unknots.  Exactly three of the ten pairs of disjoint three cycles form nontrivial links in this embedding: $(125)(346)$, $(135)(246)$ and $(136)(245)$.  All of these links are the Hopf link.  This book representation is ambient isotopic to the linear embedding of $K_6$ containing one trefoil and three Hopf links (see \cite{hj}).

\begin{figure}
\centering
\label{fig:trefoil}
\includegraphics[height = 1.5 in]{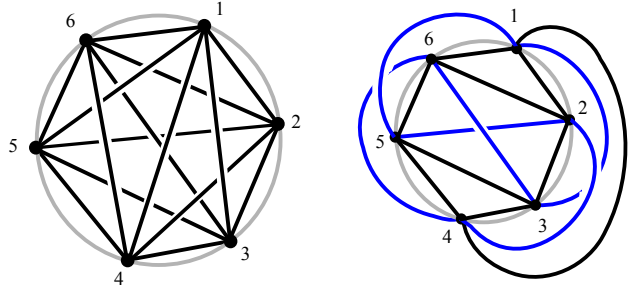}
\caption{A four-sheet book representation of $K_6$ contains exactly one trefoil.}
\end{figure}

\section{Generating $5$-sheet Book Representations of $K_6$}
\label{sec:5sheets}

Next, we find all book representations that require 5 sheets.  As in the 4-sheet case, we begin by considering the possible configurations.

\begin{lemma}
\label{lem:shtdist5}
Every 5-sheet book representation of $K_6$ can be expressed using the interior edge distribution $(3,2,2,1,1)$ or $(2,2,2,2,1)$.
\end{lemma}
\begin{proof}
Up to shifting of sheets and mirror images, these are the possible ways to distribute 9 edges to 5 sheets so that each sheet has between 1 and 3 edges:
\[(3,3,1,1,1), (3,2,2,1,1), (3,2,1,2,1), (3,2,1,1,2), (3,1,3,1,1), (3,1,2,2,1), (2,2,2,2,1)\]
As in Lemma \ref{lem:shtdist}, we note that a single edge intersects at most 4 others, so we can move edges to eliminate the configurations $(3,2,1,2,1)$, $(3,1,3,1,1)$ and $(3,1,2,2,1)$.

Consider the configuration $(3,3,1,1,1)$.  If the edge in sheet 5 does not intersect all the edges in sheet 1, then an edge can be moved into sheet 5 and by reflecting and shifting sheets, the representation can be expressed in the $(3,2,2,1,1)$ configuration.  If the edge in sheet 5 does intersect all the edges in sheet 1 as well as the edge in sheet 4, then it intersects four edges, so must be a long edge.  It follows that the other long edges are in sheets 1 and 4.  That forces the edge in sheet 3 to be short, and so it cannot intersect the edge in sheet 4 and all three edges in sheet 2.  Either the number of sheets can be reduced, or an edge from sheet 2 can be moved to obtain a $(3,2,2,1,1)$ configuration.

It remains to eliminate $(3,2,1,1,2)$.  By rotating vertices and taking mirror images, we can assume the top sheet is either $\{13, 15, 35\}$, $\{13, 14, 46\}$, or $\{13, 14, 15\}$.  A straightforward verification shows that all ways to assign pairs of edges to sheets 2 and 5 allow an edge to be moved to obtain a $(3,2,2,1,1)$ configuration.
\end{proof}

When enumerating book representations, we want to avoid listing those which can be expressed using a lower sheet number.  We define an \emph{anchor sequence} of length $s$ to be a sequence of $s$ internal edges, one per sheet, so that each edge in the sequence intersects the anchoring edges in adjacent sheets.  If a book representation contains an anchor sequence and vertex exchange is not possible, then no combination of rotating vertices, shifting sheets, reflecting, or moving edges will ever produce an empty sheet.  The following three lemmas will be used to show that any book representation of $K_6$ can be built using an anchor sequence as a backbone.

\begin{lemma}
\label{lem:anchorbuild}
Suppose that a book representation of an embedding of $K_6$ contains exactly one edge in sheets $1$ and $j$, and exactly two edges in sheet $i$ for all $1 < i < j$.  Then either there exists a sequence of edges $\{e_i\}$ such that $e_i$ is in sheet $i$ and $e_i$ intersects $e_{i-1}$ for all $1 < i \leq j$, or the sheet number can be reduced.
\end{lemma}
\begin{proof}
We proceed using induction.  When $j=2$, the result is trivial.  Suppose we know the lemma is valid when $2\leq k<j$.  Let $e_1$ denote the single edge in sheet 1, and let $e_j$ denote the single edge in sheet $j$.

If $e_1$ does not intersect at least one of the edges in sheet 2, then the sheet number can be reduced; therefore, we assume there is some edge $e_2$ in sheet 2 which intersects $e_1$. For some $i<j-1$, suppose we have found edges $e_1, e_2, ..., e_i$, one per sheet, so that each edge intersects the next.  If  $e_i$  does not intersect an edge in sheet $i+1$, then it can be moved to sheet $i+1$, leaving a single edge.  The inductive hypothesis applies, so either the sheet number can be reduced or the remaining edge is part of a new sequence of intersecting edges between sheet 1 and sheet $i$.  Since at most 3 edges can fit in a sheet of $K_6$, this remaining edge must intersect one of the original edges in sheet $i+1$, so we can extend our sequence.  Continue until the sequence is length $j-1$.  Note that if either of the edges in sheet $j-1$ fails to intersect $e_j$, then that edge can be moved into sheet $j$, and the inductive hypothesis applies.  Therefore, if the sheet number cannot be reduced, the sequence of intersecting edges can be extended to $\{e_1, e_2, \ldots, e_{j}\}$.
\end{proof}

\begin{lemma}
\label{lem:5anchor}
Any book representation with configuration $(2,2,2,2,1)$ contains an anchor sequence or can be expressed using less than five sheets.
\end{lemma}
\begin{proof}
Let $e_5$ be the edge in sheet 5.  At least one of the edges in sheet 1 must intersect $e_5$, or $e_5$ could be moved into sheet 1, reducing the sheet number.  Move an edge from sheet 1 to sheet 5 if necessary so that all edges in sheet 1 intersect $e_5$.  If the sheet number cannot be reduced, then some edge remaining in sheet 1 intersects an edge in sheet 2.  Label these edges as $e_1$ and $e_2$ respectively.  Similarly, at least one of the edges in sheet 4 must intersect $e_5$, and the same argument produces edges $e_3$ and $e_4$ in sheets 3 and 4 respectively so that $e_3$ intersects $e_4$ and $e_4$ intersects $e_5$.  If the edges $e_2$ and $e_3$ intersect, then the sequence $(e_1, e_2, e_3, e_4, e_5)$ forms an anchor sequence.

Suppose that $e_2$ and $e_3$ do not intersect.  Let $f_j$ denote the second edge in sheet $j$, for $1\leq j\leq 4$.  We consider two possible cases; either  $e_3$ intersects  $f_2$ or it does not.

Case 1:  Assume the edges $e_3$ and $f_2$ do not intersect.  This implies we can move edge $e_3$ to sheet 2, and by Lemma \ref{lem:anchorbuild}, either some edge in sheet 4 simultaneously intersects $f_3$ and $e_5$ or the sheet number can be reduced.  Relabel this edge as $e_4$ if necessary.  If $f_3$ intersects $e_2$, then $(e_1, e_2, f_3, e_4, e_5)$ is an anchor sequence.  Otherwise, we know that $f_3$ intersects $f_2$ and we can move edges $e_2$ and $e_3$ to sheet 3.  Lemma \ref{lem:anchorbuild} guarantees there is some edge in sheet 1 that intersects both $f_2$ and $e_5$ or the sheet number can be reduced; if necessary, relabel this edge as $e_1$.  Then $(e_1, f_2, f_3, e_4, e_5)$ forms an anchor sequence.

Case 2:  Assume $e_3$ intersects $f_2$.  If $f_2$ can be moved to sheet 1, then Lemma \ref{lem:anchorbuild} implies there is some sequence of intersecting edges from $e_2$ in sheet 2 to $e_5$ in sheet 5, which together with $e_1$ creates an anchor sequence.  If $f_2$ intersects $e_1$, then $(e_1, f_2, e_3, e_4, e_5)$ forms an anchor sequence.  Otherwise, $f_2$ intersects $f_1$.

If $f_1$ intersects $e_5$, then $(f_1, f_2, e_3, e_4, e_5)$ forms an anchor sequence. Otherwise, move $f_1$ to sheet 5 and move $f_2$ to sheet 1.  Unless $e_2$ intersects $f_3$, the sheet number can be reduced.  If $f_3$ intersects $e_4$, then $(e_1, e_2, f_3, e_4, e_5)$ forms an anchor sequence.  Otherwise, $f_3$ intersects $f_4$ (or the sheet number can be reduced).  Finally, $f_4$ must intersect one of $e_5$ or $f_1$.  If $f_4$ intersects $e_5$, then $(e_1, e_2, f_3, f_4, e_5)$ is an anchor sequence.  If not, then the sequence $(e_1, e_2, f_3, f_4, f_1, f_2, e_3, e_4, e_5)$ is a sequence of edges so that each edge intersects the next, and $e_5$ intersects $e_1$; in other words, the sequence forms an anchor sequence of length nine.  We show no such sequence is possible.

By rotating vertices, we can assume that the book representation with five sheets has $13$ in sheet 1; so either $e_1$ or $f_1$ is $13$.   The below table shows all possible anchor sequences of length nine, shifting sheets so that edge $13$ appears in sheet 1:
\[
\begin{tabular}{|c|c|c|c|c|c|c|c|c|c|c|c|c|c|c|c|c|c|}
\hline
13 & 13 & 13 & 13 & 13 & 13 & 13 & 13 & 13 & 13 & 13 & 13 & 13 & 13 & 13 & 13 & $e_1$ & $f_1$ \tabularnewline
24 & 24 & 24 & 24 & 24 & 24 & 25 & 25 & 25 & 25 & 26 & 26 & 26 & 26 & 26 & 26 & $e_2$ & $f_2$ \tabularnewline
35 & 35 & 35 & 35 & 36 & 36 & 14 & 46 & 36 & 46 & 14 & 14 & 15 & 15 & 15 & 15 & $f_3$ & $e_3$ \tabularnewline
14 & 46 & 46 & 46 & 15 & 25 & 36 & 15 & 14 & 35 & 25 & 35 & 36 & 46 & 46 & 46 & $f_4$ & $e_4$ \tabularnewline
36 & 15 & 15 & 25 & 26 & 14 & 24 & 36 & 26 & 14 & 36 & 24 & 14 & 25 & 35 & 35 & $f_1$ & $e_5$ \tabularnewline
25 & 26 & 36 & 14 & 14 & 35 & 35 & 24 & 15 & 26 & 15 & 36 & 25 & 36 & 14 & 24 & $f_2$ & $e_1$ \tabularnewline
46 & 14 & 25 & 36 & 35 & 46 & 46 & 35 & 46 & 15 & 46 & 15 & 46 & 14 & 25 & 36 & $e_3$ & $e_2$ \tabularnewline
15 & 36 & 14 & 15 & 46 & 15 & 15 & 14 & 35 & 36 & 35 & 46 & 35 & 35 & 36 & 14 & $e_4$ & $f_3$ \tabularnewline
26 & 25 & 26 & 26 & 25 & 26 & 26 & 26 & 24 & 24 & 24 & 25 & 24 & 24 & 24 & 25 & $e_5$ & $f_4$ \tabularnewline
\hline
\end{tabular}
\]
Note that regardless of the choice of length nine anchor sequence, a conflict arises in the 5-sheet book representation: either $e_1$ intersects $f_1$, $e_2$ intersects $f_2$, or $e_3$ intersects $f_3$, which contradicts the definition of $f_j$.  This completes the case when $e_3$ intersects $f_2$.

Thus, any book representation with configuration $(2,2,2,2,1)$ either has an anchor sequence or is equivalent to a representation with less than five sheets.
\end{proof}

Using the above lemmas together with the equivalences from Section \ref{sec:equiv}, we can identify all book representations of $K_6$ with five sheets.

\begin{theorem}
There are ten distinct book representations of $K_6$ with sheet number five.
\label{5sheets}
\end{theorem}

\begin{proof}
By Lemma \ref{lem:shtdist5}, it suffices to consider the configurations $(3,2,2,1,1)$ and $(2,2,2,2,1)$.  First, consider the representations that can be expressed in a $(2,2,2,2,1)$ configuration.  By Lemma \ref{lem:5anchor}, we can begin with the possible anchors of length 5.  Up to rotation and reflection, there are only three distinct anchor sequences:  $\{13, 24, 35, 46, 25\}$, $\{13, 24, 35, 14, 25\}$, and $\{13, 24, 36, 14, 25\}$.  Enumerating all possible ways to insert the extra edges and then checking for equivalences using the moves from Section \ref{sec:equiv} results in the eight representations 5s1, 5s2, 5s3, 5s4, and their mirror images.  See Section \ref{appendix} for a list of edges in each sheet and knots and links contained in each representation.  These representations must be distinct since the number of Hopf links, left trefoils, or right trefoils differs between each pair.

Next, consider the configuration of edges $(3,2,2,1,1)$.  We assume the edge in sheet 5 intersects all edges in sheet 1; otherwise, the representation can be expressed in a $(2,2,2,2,1)$ configuration. That implies that a long edge is in sheet 5, and by rotation of vertices we may assume this edge is $25$.  The top sheet either contains $\{13, 14, 46\}$ or its reflection $\{13, 36, 46\}$, obtained by reflecting across the diameter $25$.    Up to reflection, we can assume the top sheet is $\{13, 14, 46\}$ which forces edge $36$ to be in sheet 4.  Sheet 3 must consist of a pair non-intersecting edges at least one of which intersects $36$, so it must be one of $\{15,24\}, \{15, 35\},$ or $\{24, 26\}$.  The choice made determines sheet 2.  In the second and third options, edges can be moved to obtain a $(2,2,2,2,1)$ representation; so the only 5-sheet representations that cannot be expressed in a  $(2,2,2,2,1)$ configuration is equivalent to the following or its mirror image:
\[\begin{tabular}{|c|c|}
\hline
Sheet 1 & 13, 14, 46\tabularnewline
\hline
Sheet 2 & 26, 35\tabularnewline
\hline
Sheet 3 & 15, 24\tabularnewline
\hline
Sheet 4 & 36\tabularnewline
\hline
Sheet 5 & 25    \tabularnewline
\hline
\end{tabular}
\]

This appears as representation 5s5 in Section \ref{appendix}. Like representation 5s4, this embedding contains exactly six trefoils and five Hopf links.  (See Section \ref{appendix}.)  Note that edge $36$ appears in all six of the knotted cycles.  Since no edge in 5s4 appears in all knotted cycles, there is no way to construct an ambient isotopy between the two book representations or their mirror images.
\end{proof}

Book representations with sheet number 5 contain either three or five Hopf links; two, three, or four knotted Hamiltonian cycles (all trefoils); and either one or two knotted 5-cycles.

\section{Generating Book Representations of $K_6$ with six or more sheets}
\label{sec:moresheets}
Enumerating the book representations with six through nine sheets can be approached in a similar manner.  First, consider the set of possible configurations.
The following lemma helps to reduce the number of configurations which need to be considered:

\begin{lemma} Suppose a book representation of $K_6$ has a sheet with three interior edges between two sheets containing one edge each.  Then at least one of three interior edges in the middle sheet can be moved to an adjacent sheet.
\label{lem:no3sheet}
\end{lemma}
\begin{proof}  After rotating vertices and reflecting as needed, we may assume the sheet with 3 interior edges contains either $\{13, 14, 46\}$, $\{13, 14, 15\}$, or $\{13, 15, 35\}$.  The edge $25$ is the only edge which intersects all three of the edges $13$, $14$, and $46$.  The edge $26$ is the only edge which intersects all of $13$, $14$, and $15$.  No edge intersects all three of the edges $13$, $15$, and $35$.  Therefore, if each of the adjacent sheets contains only one edge, some edge from the middle sheet can always be moved. \end{proof}

In fact, we will show that every configuration with six or more sheets is equivalent to one which has at most two edges per sheet, so that Lemma \ref{lem:anchorbuild} guarantees the existence of an anchor sequence.

Next, we list the set of possible non-equivalent anchor sequences in each case.  Note that each short edge intersects two other short edges and one long edge.   Let $L$ denote a long edge, chosen from the set of edges $\{14, 25,36\}$, and let $S$ denote a short edge, chosen from the set of edges $\{13, 15, 24, 26, 35, 46\}$).  Consider the possible sequences of long and short edges.  The sub-sequence $LSL$ cannot appear in an anchor sequence, since each edge of length two intersects a unique edge of length three.  Similarly, $LSSSSL$ cannot appear, since the first and fourth consecutive short edges intersect the same long edge.  Since every anchor sequence of length five or more must contain two consecutive short edges, we shift sheets to obtain the longest possible subsequence of short edges at the start.  Up to reflection and rotation, we may assume all anchor sequences contain $13$ in sheet 1 and $24$ in sheet 2.   If an anchor sequence contains any long edges, its last edge will be edge $25$; otherwise, the last edge is edge $26$.

After determining the possible anchor sequences, we build book representations by inserting the additional edges not contained in the anchor sequence and check for equivalences among the set of possible representations using the moves from Section \ref{sec:equiv}.  Finally, we consider the collection of knots and links contained within the representations to ensure that the remaining representations are distinct.

Below, we follow these steps to determine the number of distinct book representations for sheet numbers six through nine.

\begin{theorem}
There are 20 distinct book representations of $K_6$ with sheet number six.
\end{theorem}
\begin{proof} Due to Lemma \ref{lem:no3sheet}, the only possible configuration we need to consider with three edges in a sheet is $(3,2,1,1,1,1)$.  Suppose no edge from the top sheet can be bumped to sheet 6, and the sheet number cannot be reduced.  By rotating vertex labels, we may assume the top sheet is either $\{13, 14, 46\}$ or $\{13, 14, 15\}$.  If the top sheet is $\{13, 14, 46\}$, sheet 6 must contain edge $25$ which forces sheet 5 to contain edge $36$.  Sheet 4 can either contain edge $15$, which requires edge $26$ in sheet 3, or sheet 4 can contain edge $24$, which requires edge $35$ in sheet 3.  In either case, the remaining pair of edges intersects, so cannot be assigned to sheet 2.  If the top sheet is $\{13, 14, 15\}$, then edge $26$ is in sheet 6, but there is no other edge which intersects edge $26$.  Therefore, a $(3,2,1,1,1,1)$ configuration is either equivalent to a $(2,2,1,1,1,2)$ configuration, or can be reduced to a lower sheet number.

A configuration with at most two edges per sheet contains an anchor sequence.  The possible non-equivalent anchor sequences of length 6, up to reflection, are $(13, 24, 35, 14, 36, 25)$, $(13, 24, 35, 46, 15, 26)$, and $(13, 24, 36, 15, 46, 25)$.  The missing three edges for each possible anchor can be inserted in multiple different ways; however, using the equivalences in Section \ref{sec:equiv}, each of them is equivalent to one of the ten six-sheet book representations shown in Section \ref{appendix}, or its mirror image.  No two of the six-sheet representations have the same number of links and knots of each type; therefore, no pair can be equivalent.  The set of knots and links for each representation is also distinct from those obtained in the lower sheet representations, guaranteeing the sheet number cannot be reduced.  Furthermore, in each representation listed the number of Hamiltonian cycles which are right handed trefoils is different from the number of Hamiltonian cycles which are left-handed trefoils--so it is impossible for any six-sheet representation to be equivalent to is mirror image.  Therefore, there are exactly 20 distinct six-sheet book representations, as claimed.
\end{proof}

Note that six is the minimum number of sheets required to obtain a four-crossing knot or link--book representations of $K_6$ with six sheets can contain a figure-eight knot $4_1$ or a Solomon link $4^2_1$ (but not both).

\begin{theorem}
There are ten distinct book representations of $K_6$ with sheet number seven.
\end{theorem}
\begin{proof}
Since every configuration with seven sheets has at most two edges per sheet, an anchor sequence exists.  The anchor sequence must contain at least one long edge, so by shifting sheets, rotating vertices, and reflecting, we may assume that the anchor sequence begins with edges $13$ and $24$ and ends with $25$.  Up to equivalence and reflection, the possible length seven anchors are $(13, 24, 35, 46, 15, 36, 25)$ and $(13, 24, 36, 14, 35, 46, 25)$.  After inserting the additional two edges in the anchor sequences and checking for equivalences, we obtain one of the five seven-sheet book representations in Section \ref{appendix} (or a mirror image).  No two of the seven-sheet representations have the same number of links and knots of each type; therefore, no pair can be equivalent.  The set of knots and links for each representation is also distinct from those obtained in the lower sheet representations, guaranteeing the sheet number cannot be reduced.  Each embedding except 7s2 has a different number of left and right handed trefoils, so is necessarily distinct from its mirror image.  A further justification is needed to show that 7s2 is distinct from its mirror image.

The embedding 7s2 and its mirror image 7s2* are shown below; both contain four knotted Hamiltonian cycles (two left and two right trefoils) and two knotted 5-cycles (one left and one right trefoil).
\[
\begin{tabular}{|l|l|l|}
        \hline
         & 7s2 & 7s2* \tabularnewline
        \hline
        \hline
        Sheet 1 & 13, 14 & 13, 14 \tabularnewline
        \hline
        Sheet 2 & 24, 26 & 25  \tabularnewline
        \hline
        Sheet 3 & 35 & 36 \tabularnewline
        \hline
        Sheet 4 & 46 & 15 \tabularnewline
        \hline
        Sheet 5 & 15 & 46 \tabularnewline
        \hline
        Sheet 6 & 36 & 35 \tabularnewline
        \hline
        Sheet 7 & 25 & 24, 26 \tabularnewline
        \hline
        Left trefoils & (124635) & (136425)\tabularnewline
        & (135246) & (143625) \tabularnewline
        & (14635) &  (13625)\tabularnewline
        \hline
        Right trefoils & (136425) & (124635) \tabularnewline
        & (143625) & (135246)\tabularnewline
        & (13625)  &  (14635)\tabularnewline
        \hline
\end{tabular}
\]
Note that edge $14$ in 7s2 is contained in the 5-cycle left-trefoil $(14635)$ but not in either of the left-trefoil Hamiltonian cycles; however, each edge in the 5-cycle left-trefoil $(13625)$ in 7s2* is contained in at least one of the left-trefoil Hamiltonian cycles.  Therefore, no ambient isotopy is possible between 7s2 and 7s2*, and the total number of distinct seven-sheet book representations is ten, as claimed.
\end{proof}

Note that every seven-sheet representation contains at least four pairs of linked 3-cycles and at least five knotted cycles.

\begin{theorem}
There are 12 distinct book representations of $K_6$ with sheet number eight.
\end{theorem}

\begin{proof}
An anchor sequence of length eight omits exactly one edge.   Up to equivalence and reflection, the possible length eight anchors are:
\[(13, 24, 35, 14, 26, 15, 46, 25)\]\[(13, 24, 35, 14, 36, 15, 46, 25)\]\[(13, 24, 35, 46, 15, 26, 14, 25)\] \[(13, 24, 35, 46, 15, 36, 14, 25)\]
Inserting the missing edge in all possible ways and checking for equivalences results in one of the six eight-sheet representations listed in Section \ref{appendix}, or a mirror image.  No two of the eight-sheet representations have the same number of links and knots of each type; therefore, no pair can be equivalent.  The set of knots and links for each representation is also distinct from those obtained in the lower sheet representations, guaranteeing the sheet number cannot be reduced.  Furthermore, in each representation listed the number of cycles which are right handed trefoils is different from the number of cycles which are left-handed trefoils--so it is impossible for any eight-sheet representation to be equivalent to is mirror image.  Therefore, there are exactly 12 distinct eight-sheet book representations, as claimed.
\end{proof}

\begin{theorem}
There are four distinct book representations of $K_6$ with sheet number nine.
\end{theorem}

\begin{proof}
In order to construct a book representation with maximal sheet number, each sheet must contain one internal edge.  Each anchor sequence of length nine generates a book representation of $K_6$ with nine sheets.  There are only two non-equivalent anchor sequences, up to reflection:
\[(13, 24, 35, 46, 15, 26, 14, 36, 25)\] and \[(13, 24, 36, 15, 26, 14, 35, 46, 25).\]
These and their mirror images are the four non-equivalent book representations.  Because neither contains the same number of right trefoils as left trefoils, neither is equivalent to its mirror image.  Furthermore, each contains a different number of links and knots of each type from each other and from the lower sheet representations.  This establishes that there are four distinct book representations as claimed.
\end{proof}

\section{Table of Links and Knots in Book Representations of $K_6$}
\label{appendix}
This section contains a list of the book representations of $K_6$, up to mirror image, and lists the knotted cycles and linked 3-cycles that can be found in each.  The 3-sheet representation is achiral; all others are distinct from their mirror images.

The numbering convention lists the number of sheets, followed by a ranking of book representations ordered by number of links, then by number of knotted cycles; so $K_6: nsm$ denotes the $m$th  $n$-sheet book representation of $K_6$.

The knotted cycles were identified via their Dowker-Thistlethwaite notation using KnotScape together with code written by David Toth, Michael Walton, and and the author \cite{knotscape, toth}.  The handedness of the trefoils and the linked cycles were determined by hand by the author.

\subsection*{3 sheets} 

\noindent $K_6:3s1$ \vspace{1 em}

  \begin{minipage}[b]{0.35\linewidth}
   \begin{tabular}{|l|l|}
        \hline
        Sheet Number & Edges\tabularnewline
        \hline
        \hline
        Sheet 1 & 13, 14, 46\tabularnewline
        \hline
        Sheet 2 & 15, 24, 25\tabularnewline
        \hline
        Sheet 3 & 26, 35, 36\tabularnewline
        \hline
        \end{tabular}
    \par\vspace{0pt}
  \end{minipage}
  \begin{minipage}[b]{0.60\linewidth}
    \begin{tabular}{|l|}
        \hline
        1 Hopf link\tabularnewline
        \hline
        (135)(246)\tabularnewline
        \hline
        \end{tabular}
    \par\vspace{0pt}
  \end{minipage}%

\subsection*{4 sheets} 

\noindent $K_6:4s1$
\vspace{1 em}

 \begin{minipage}[b]{0.35\linewidth}
  \begin{tabular}{|l|l|}
        \hline
        Sheet Number & Edges\tabularnewline
        \hline
        \hline
        Sheet 1 & 13, 14, 46\tabularnewline
        \hline
        Sheet 2 & 15, 24\tabularnewline
        \hline
        Sheet 3 & 26, 36\tabularnewline
        \hline
        Sheet 4 & 25, 35\tabularnewline
        \hline
        \end{tabular}
    \par\vspace{0pt}
  \end{minipage}
  \begin{minipage}[b]{0.60\linewidth}
    \begin{tabular}{|l|l|}
        \hline
        3 Hopf links & 1 trefoil \tabularnewline
        \hline
        (125)(346) & (136425)R \tabularnewline

        (135)(246) & \tabularnewline

        (136)(245) & \tabularnewline
        \hline
        \end{tabular}
    \par\vspace{0pt}
  \end{minipage}

  \subsection*{5 sheets}  
\noindent $K_6:5s1$ \vspace{1 em}

 \begin{minipage}[b]{0.35\linewidth}
  \begin{tabular}{|l|l|}
        \hline
        Sheet Number & Edges\tabularnewline
        \hline
        \hline
        Sheet 1 & 13, 36, 46\tabularnewline
        \hline
        Sheet 2 & 24, 15\tabularnewline
        \hline
        Sheet 3 & 35 \tabularnewline
        \hline
        Sheet 4 & 14\tabularnewline
        \hline
        Sheet 5 & 25, 26\tabularnewline
        \hline
        \end{tabular}
    \par\vspace{0pt}
  \end{minipage}
  \begin{minipage}[b]{0.60\linewidth}
    \begin{tabular}{|l|l|}
        \hline
        3 Hopf links & 3 trefoils \tabularnewline
        \hline
        (124)(356) &  (135624)L \tabularnewline
        (135)(246) &  (142536)L \tabularnewline \cline{2-2}
        (146)(235) &  (13524)L  \tabularnewline
        \hline
        \end{tabular}
    \par\vspace{0pt}
  \end{minipage}

\vspace{1 em}
\noindent $K_6:5s2$

\vspace{1 em}
 \begin{minipage}[b]{0.35\linewidth}
\begin{tabular}{|l|l|}
        \hline
        Sheet Number & Edges\tabularnewline
        \hline
        \hline
        Sheet 1 & 13, 14\tabularnewline
        \hline
        Sheet 2 & 24, 26\tabularnewline
        \hline
        Sheet 3 & 15, 35\tabularnewline
        \hline
        Sheet 4 & 36, 46\tabularnewline
        \hline
        Sheet 5 & 25    \tabularnewline
        \hline
\end{tabular}
    \par\vspace{0pt}
  \end{minipage}
  \begin{minipage}[b]{0.60\linewidth}
    \begin{tabular}{|l|l|}
        \hline
        3 Hopf links &  5 trefoils\tabularnewline
        \hline
       (136)(245) &  (134625)R \tabularnewline
       (145)(236) &(135246)L \tabularnewline
       (146)(235) &(143625)R \tabularnewline \cline{2-2}
                &(13625)R \tabularnewline
                &(14625)R \tabularnewline
        \hline
        \end{tabular}
    \par\vspace{0pt}
  \end{minipage}
\vspace{1 em}

\noindent $K_6:5s3$
\vspace{1 em}

 \begin{minipage}[b]{0.35\linewidth}
\begin{tabular}{|l|l|}
        \hline
        Sheet Number & Edges\tabularnewline
        \hline
        \hline
        Sheet 1 & 13, 36\tabularnewline
        \hline
        Sheet 2 & 15, 24\tabularnewline
        \hline
        Sheet 3 & 35\tabularnewline
        \hline
        Sheet 4 & 14, 46\tabularnewline
        \hline
        Sheet 5 & 25, 26 \tabularnewline
        \hline
\end{tabular}
    \par\vspace{0pt}
  \end{minipage}
  \begin{minipage}[b]{0.60\linewidth}
    \begin{tabular}{|l|l|}
        \hline
        3 Hopf links & 5 trefoils \tabularnewline
        \hline
        (124)(356) & (135246)L \tabularnewline
        (125)(346) & (135624)L \tabularnewline
        (135)(246) & (142536)L \tabularnewline \cline{2-2}
                    & (13524)L \tabularnewline
                    & (24635)L \tabularnewline
        \hline
        \end{tabular}
    \par\vspace{0pt}
  \end{minipage}
\vspace{1 em}

  \noindent $K_6:5s4$
\vspace{1 em}

 \begin{minipage}[b]{0.35\linewidth}
\begin{tabular}{|l|l|}
        \hline
        Sheet Number & Edges\tabularnewline
        \hline
        \hline
        Sheet 1 & 13, 36\tabularnewline
        \hline
        Sheet 2 & 15, 24\tabularnewline
        \hline
        Sheet 3 & 35, 26\tabularnewline
        \hline
        Sheet 4 & 14, 46\tabularnewline
        \hline
        Sheet 5 & 25 \tabularnewline
        \hline
\end{tabular}
    \par\vspace{0pt}
  \end{minipage}
  \begin{minipage}[b]{0.60\linewidth}
    \begin{tabular}{|l|l|}
        \hline
        5 Hopf links & 6 trefoils \tabularnewline
        \hline
        (124)(356) & (135246)L \tabularnewline
        (125)(346) & (142536)L \tabularnewline
        (134)(256) & (142635)R \tabularnewline
        (135)(246) & (143625)L \tabularnewline \cline{2-2}
        (145)(236) & (13524)L \tabularnewline
                    & (24635)L \tabularnewline
        \hline
        \end{tabular}
    \par\vspace{0pt}
  \end{minipage}
\vspace{1 em}

  \noindent $K_6:5s5$
\vspace{1 em}

\begin{minipage}[b]{0.35\linewidth}
\begin{tabular}{|l|l|}
        \hline
        Sheet Number & Edges\tabularnewline
        \hline
        \hline
        Sheet 1 & 13, 36, 46\tabularnewline
        \hline
        Sheet 2 & 15, 24\tabularnewline
        \hline
        Sheet 3 & 26, 35\tabularnewline
        \hline
        Sheet 4 & 14\tabularnewline
        \hline
        Sheet 5 & 25    \tabularnewline
        \hline
\end{tabular}
    \par\vspace{0pt}
  \end{minipage}
  \begin{minipage}[b]{0.60\linewidth}
    \begin{tabular}{|l|l|}
        \hline
        5 Hopf links & 6 trefoils \tabularnewline
        \hline
        (124)(356) & (135264)L \tabularnewline
        (134)(256) & (142536)L \tabularnewline
        (135)(246) & (142635)R \tabularnewline
        (145)(236) & (143625)L \tabularnewline \cline{2-2}
        (146)(235) & (13524)L \tabularnewline
                    & (14625)L \tabularnewline
        \hline
        \end{tabular}
    \par\vspace{0pt}
  \end{minipage}

\subsection*{6 sheets}  
  \noindent $K_6:6s1$
\vspace{1 em}

\begin{minipage}[b]{0.35\linewidth}
\begin{tabular}{|l|l|}
        \hline
        Sheet Number & Edges\tabularnewline
        \hline
        \hline
        Sheet 1 & 13, 14 \tabularnewline
        \hline
        Sheet 2 & 24, 25  \tabularnewline
        \hline
        Sheet 3 & 35, 36 \tabularnewline
        \hline
        Sheet 4 & 46 \tabularnewline
        \hline
        Sheet 5 & 15 \tabularnewline
        \hline
        Sheet 6 & 26 \tabularnewline
        \hline
\end{tabular}
    \par\vspace{0pt}
  \end{minipage}
  \begin{minipage}[b]{0.60\linewidth}
    \begin{tabular}{|l|l|l|}
        \hline
        1 Hopf link & 1 Solomon & 6 trefoils \tabularnewline
        \hline
        (145)(236)& (135)(246) & (134625)L \tabularnewline
        &   &   (136245)L \tabularnewline
        &   &   (143625)L \tabularnewline
        &   &   (146235)L \tabularnewline
        &   &   (13625)L   \tabularnewline \cline{3-3}
        &   &   (14625)L    \tabularnewline
        \hline
        \end{tabular}
    \par\vspace{0pt}
  \end{minipage}
\vspace{1 em}

\noindent $K_6:6s2$
\vspace{1 em}

\begin{minipage}[b]{0.35\linewidth}
\begin{tabular}{|l|l|}
        \hline
        Sheet Number & Edges\tabularnewline
        \hline
        \hline
        Sheet 1 & 13, 14 \tabularnewline
        \hline
        Sheet 2 & 24, 26 \tabularnewline
        \hline
        Sheet 3 & 35, 36 \tabularnewline
        \hline
        Sheet 4 & 15 \tabularnewline
        \hline
        Sheet 5 & 46 \tabularnewline
        \hline
        Sheet 6 & 25 \tabularnewline
        \hline
\end{tabular}
    \par\vspace{0pt}
  \end{minipage}
  \begin{minipage}[b]{0.60\linewidth}
    \begin{tabular}{|l|l|l|}
        \hline
        3 Hopf links &  3 trefoils & 1 figure-eight \tabularnewline
        \hline
        (125)(346) & (134625)R & (136425) \tabularnewline
        (136)(245) & (135246)L & \tabularnewline
        (146)(235) & (146325)R & \tabularnewline
        \hline
        \end{tabular}
    \par\vspace{0pt}
  \end{minipage}
\vspace{1 em}

  \noindent $K_6:6s3$
\vspace{1 em}

\begin{minipage}[b]{0.35\linewidth}
\begin{tabular}{|l|l|}
        \hline
        Sheet Number & Edges\tabularnewline
        \hline
        \hline
        Sheet 1 & 13, 14 \tabularnewline
        \hline
        Sheet 2 & 24 \tabularnewline
        \hline
        Sheet 3 & 36 \tabularnewline
        \hline
        Sheet 4 & 15 \tabularnewline
        \hline
        Sheet 5 & 26, 46 \tabularnewline
        \hline
        Sheet 6 & 25, 35 \tabularnewline
        \hline
\end{tabular}
    \par\vspace{0pt}
  \end{minipage}
  \begin{minipage}[b]{0.60\linewidth}
    \begin{tabular}{|l|l|l|}
        \hline
         3 Hopf links &  4 trefoils & 1 figure-eight \tabularnewline
        \hline
        (125)(346) & (124635)R & (136425) \tabularnewline
        (136)(245) & (136245)L &\tabularnewline
        (145)(236) & (146325)R &\tabularnewline
        & (14635)R  & \tabularnewline \cline{2-2}
        \hline
        \end{tabular}
    \par\vspace{0pt}
  \end{minipage}
\vspace{1 em}

\noindent $K_6:6s4$ 
\vspace{1 em}

\begin{minipage}[b]{0.35\linewidth}
\begin{tabular}{|l|l|}
        \hline
        Sheet Number & Edges\tabularnewline
        \hline
        \hline
        Sheet 1 & 13 \tabularnewline
        \hline
        Sheet 2 & 24, 26  \tabularnewline
        \hline
        Sheet 3 & 35  \tabularnewline
        \hline
        Sheet 4 & 14, 15 \tabularnewline
        \hline
        Sheet 5 & 36, 46 \tabularnewline
        \hline
        Sheet 6 & 25 \tabularnewline
        \hline
\end{tabular}
    \par\vspace{0pt}
  \end{minipage}
  \begin{minipage}[b]{0.60\linewidth}
    \begin{tabular}{|l|l|}
        \hline
        3 Hopf links & 5 trefoils \tabularnewline
        \hline
        (124)(356) & (134625)R \tabularnewline
        (134)(256) & (135246)L \tabularnewline
        (136)(245) & (136254)R \tabularnewline  \cline{2-2}
        & (13524)L \tabularnewline
        & (13625)R \tabularnewline
        \hline
        \end{tabular}
    \par\vspace{0pt}
  \end{minipage}
\vspace{1 em}

  \noindent $K_6:6s5$  
\vspace{1 em}

\begin{minipage}[b]{0.35\linewidth}
\begin{tabular}{|l|l|}
        \hline
        Sheet Number & Edges\tabularnewline
        \hline
        \hline
        Sheet 1 & 13, 14 \tabularnewline
        \hline
        Sheet 2 & 24, 26 \tabularnewline
        \hline
        Sheet 3 & 36 \tabularnewline
        \hline
        Sheet 4 & 15 \tabularnewline
        \hline
        Sheet 5 & 46 \tabularnewline
        \hline
        Sheet 6 & 25, 35 \tabularnewline
        \hline
\end{tabular}
    \par\vspace{0pt}
  \end{minipage}
  \begin{minipage}[b]{0.60\linewidth}
    \begin{tabular}{|l|l|l|}
        \hline
        3 Hopf links & 6 trefoils & 1 figure-eight \tabularnewline
        \hline
        (125)(346) & (124635)R & (136425)\tabularnewline
        (135)(246) & (134625)R & \tabularnewline
        (136)(245) & (146235)R & \tabularnewline
        & (146325)R &\tabularnewline \cline{2-2}
        & (14625)R & \tabularnewline
        & (14635)R & \tabularnewline
        \hline
        \end{tabular}
    \par\vspace{0pt}
  \end{minipage}
\vspace{1 em}

  \noindent $K_6:6s6$
\vspace{1 em}

\begin{minipage}[b]{0.35\linewidth}
\begin{tabular}{|l|l|}
        \hline
        Sheet Number & Edges\tabularnewline
        \hline
        \hline
        Sheet 1 & 13, 14 \tabularnewline
        \hline
        Sheet 2 & 24 \tabularnewline
        \hline
        Sheet 3 & 35, 36 \tabularnewline
        \hline
        Sheet 4 & 46 \tabularnewline
        \hline
        Sheet 5 & 15, 25 \tabularnewline
        \hline
        Sheet 6 & 26 \tabularnewline
        \hline
\end{tabular}
    \par\vspace{0pt}
  \end{minipage}
  \begin{minipage}[b]{0.60\linewidth}
    \begin{tabular}{|l|l|l|}
        \hline
        3 Hopf links & 1 Solomon & 3 trefoils \tabularnewline
        \hline
        (136)(245) & (135)(246) & (135246)L \tabularnewline
        (145)(236) &            & (136245)L \tabularnewline
        (146)(235) &            & (146235)L \tabularnewline
        \hline
        \end{tabular}
    \par\vspace{0pt}
  \end{minipage}
\vspace{1 em}

\noindent $K_6:6s7$
\vspace{1 em}

\begin{minipage}[b]{0.35\linewidth}
\begin{tabular}{|l|l|}
        \hline
        Sheet Number & Edges\tabularnewline
        \hline
        \hline
        Sheet 1 & 13, 46 \tabularnewline
        \hline
        Sheet 2 & 24  \tabularnewline
        \hline
        Sheet 3 & 15, 35 \tabularnewline
        \hline
        Sheet 4 & 14 \tabularnewline
        \hline
        Sheet 5 & 26, 36 \tabularnewline
        \hline
        Sheet 6 & 25 \tabularnewline
        \hline
\end{tabular}
    \par\vspace{0pt}
  \end{minipage}
  \begin{minipage}[b]{0.60\linewidth}
    \begin{tabular}{|l|l|}
        \hline
        5 Hopf links & 4 trefoils \tabularnewline
        \hline
        (124)(356) & (125364)R \tabularnewline
        (125)(346) & (135624)L \tabularnewline
        (135)(246) & (136425)R \tabularnewline \cline{2-2}
        (136)(245) & (13524)L  \tabularnewline
        (146)(235) & \tabularnewline
        \hline
        \end{tabular}
    \par\vspace{0pt}
  \end{minipage}
\vspace{1 em}

  \noindent $K_6:6s8$ 
\vspace{1 em}

\begin{minipage}[b]{0.35\linewidth}
\begin{tabular}{|l|l|}
        \hline
        Sheet Number & Edges\tabularnewline
        \hline
        \hline
        Sheet 1 & 13, 46 \tabularnewline
        \hline
        Sheet 2 & 24, 26  \tabularnewline
        \hline
        Sheet 3 & 15, 35 \tabularnewline
        \hline
        Sheet 4 & 14 \tabularnewline
        \hline
        Sheet 5 & 36 \tabularnewline
        \hline
        Sheet 6 & 25 \tabularnewline
        \hline
\end{tabular}
    \par\vspace{0pt}
  \end{minipage}
  \begin{minipage}[b]{0.60\linewidth}
    \begin{tabular}{|l|l|}
        \hline
        5 Hopf links & 4 trefoils \tabularnewline
        \hline
        (124)(356) & (125364)R \tabularnewline
        (125)(346) & (135264)L \tabularnewline
        (134)(256) & (136254)R \tabularnewline
        (136)(245) & (136425)R \tabularnewline
        (146)(235) & \tabularnewline
        & \tabularnewline
        \hline
        \end{tabular}
    \par\vspace{0pt}
  \end{minipage}
\vspace{1 em}

  \noindent $K_6:6s9$ 
\vspace{1 em}

\begin{minipage}[b]{0.35\linewidth}
\begin{tabular}{|l|l|}
        \hline
        Sheet Number & Edges\tabularnewline
        \hline
        \hline
        Sheet 1 & 13, 14 \tabularnewline
        \hline
        Sheet 2 & 24 \tabularnewline
        \hline
        Sheet 3 & 35, 36 \tabularnewline
        \hline
        Sheet 4 & 15 \tabularnewline
        \hline
        Sheet 5 & 26, 46 \tabularnewline
        \hline
        Sheet 6 & 25 \tabularnewline
        \hline
\end{tabular}
    \par\vspace{0pt}
  \end{minipage}
  \begin{minipage}[b]{0.60\linewidth}
    \begin{tabular}{|l|l|l|}
        \hline
        5 Hopf links & 3 trefoils & 1 figure-eight \tabularnewline
        \hline
        (125)(346) & (135246)L & (136425)\tabularnewline
        (135)(246) & (136245)L & \tabularnewline
        (136)(245) & (146325)R & \tabularnewline
        (145)(236) & &\tabularnewline
        (146)(235) & &\tabularnewline
        \hline
        \end{tabular}
    \par\vspace{0pt}
  \end{minipage}
\vspace{1 em}

  \noindent $K_6:6s10$ 
\vspace{1 em}

\begin{minipage}[b]{0.35\linewidth}
\begin{tabular}{|l|l|}
        \hline
        Sheet Number & Edges\tabularnewline
        \hline
        \hline
        Sheet 1 & 13, 46 \tabularnewline
        \hline
        Sheet 2 & 15, 24  \tabularnewline
        \hline
        Sheet 3 & 26, 35 \tabularnewline
        \hline
        Sheet 4 & 14 \tabularnewline
        \hline
        Sheet 5 & 36 \tabularnewline
        \hline
        Sheet 6 & 25 \tabularnewline
        \hline
\end{tabular}
    \par\vspace{0pt}
  \end{minipage}
  \begin{minipage}[b]{0.60\linewidth}
    \begin{tabular}{|l|l|}
        \hline
        7 Hopf links & 7 trefoils \tabularnewline
        \hline
        (124)(356)& (125364)R \tabularnewline
        (125)(346)& (135264)L \tabularnewline
        (134)(256)& (136254)R \tabularnewline
        (135)(246)& (136425)R \tabularnewline
        (136)(245)& (142635)R \tabularnewline \cline{2-2}
        (145)(236)& (13524)L \tabularnewline
        (146)(235)& (14625)L \tabularnewline
        \hline
        \end{tabular}
    \par\vspace{0pt}
  \end{minipage}

  \subsection*{7 sheets}  

\noindent $K_6:7s1$  
\vspace{1 em}

\begin{minipage}[b]{0.35\linewidth}
\begin{tabular}{|l|l|}
        \hline
        Sheet Number & Edges\tabularnewline
        \hline
        \hline
        Sheet 1 & 13, 14 \tabularnewline
        \hline
        Sheet 2 & 24 \tabularnewline
        \hline
        Sheet 3 & 35 \tabularnewline
        \hline
        Sheet 4 & 46 \tabularnewline
        \hline
        Sheet 5 & 15 \tabularnewline
        \hline
        Sheet 6 & 26, 36 \tabularnewline
        \hline
        Sheet 7 & 25 \tabularnewline
        \hline
\end{tabular}
    \par\vspace{0pt}
  \end{minipage}
  \begin{minipage}[b]{0.60\linewidth}
    \begin{tabular}{|l|l|l|}
        \hline
        3 Hopf links & 1 Solomon & 5 trefoils \tabularnewline
        \hline
        (125)(346) & (135)(246) & (124635)L \tabularnewline
        (136)(245) &  & (135246)L \tabularnewline
        (146)(235) & & (136425)R \tabularnewline
        & & (146235)L \tabularnewline \cline{3-3}
        & & (14635)L \tabularnewline
        \hline
        \end{tabular}
    \par\vspace{0pt}
  \end{minipage}
\vspace{1 em}

\noindent $K_6:7s2$ 
\vspace{1 em}

\begin{minipage}[b]{0.35\linewidth}
\begin{tabular}{|l|l|}
        \hline
        Sheet Number & Edges\tabularnewline
        \hline
        \hline
        Sheet 1 & 13, 14 \tabularnewline
        \hline
        Sheet 2 & 24, 26  \tabularnewline
        \hline
        Sheet 3 & 35 \tabularnewline
        \hline
        Sheet 4 & 46 \tabularnewline
        \hline
        Sheet 5 & 15 \tabularnewline
        \hline
        Sheet 6 & 36 \tabularnewline
        \hline
        Sheet 7 & 25 \tabularnewline
        \hline
\end{tabular}
    \par\vspace{0pt}
  \end{minipage}
  \begin{minipage}[b]{0.60\linewidth}
    \begin{tabular}{|l|l|}
        \hline
         5 Hopf links & 6 trefoils \tabularnewline
        \hline
        (125)(346) & (124635)L \tabularnewline
        (135)(246) & (135246)L \tabularnewline
        (136)(245) & (136425)R \tabularnewline
        (145)(236) & (143625)R \tabularnewline \cline{2-2}
        (146)(235) & (13625)R \tabularnewline
         & (14635)L \tabularnewline
        \hline
        \end{tabular}
    \par\vspace{0pt}
  \end{minipage}
\vspace{1 em}

  \noindent $K_6:7s3$ 
\vspace{1 em}

\begin{minipage}[b]{0.35\linewidth}
\begin{tabular}{|l|l|}
        \hline
        Sheet Number & Edges\tabularnewline
        \hline
        \hline
        Sheet 1 & 13, 15 \tabularnewline
        \hline
        Sheet 2 & 24\tabularnewline
        \hline
        Sheet 3 & 36 \tabularnewline
        \hline
        Sheet 4 & 14 \tabularnewline
        \hline
        Sheet 5 & 26, 35 \tabularnewline
        \hline
        Sheet 6 & 46 \tabularnewline
        \hline
        Sheet 7 & 25 \tabularnewline
        \hline
\end{tabular}
    \par\vspace{0pt}
  \end{minipage}
  \begin{minipage}[b]{0.60\linewidth}
    \begin{tabular}{|l|l|}
        \hline
        5 Hopf links & 6 trefoils \tabularnewline
        \hline
        (124)(356) & (125364)L \tabularnewline
        (135)(246) & (135246)L \tabularnewline
        (136)(245) & (136524)L \tabularnewline
        (145)(236) & (142635)L \tabularnewline \cline{2-2}
        (146)(235) & (13624)L \tabularnewline
         & (14635)L \tabularnewline
        \hline
        \end{tabular}
    \par\vspace{0pt}
  \end{minipage}
\vspace{1 em}

  \noindent $K_6:7s4$  
\vspace{1 em}

\begin{minipage}[b]{0.35\linewidth}
\begin{tabular}{|l|l|}
        \hline
        Sheet Number & Edges\tabularnewline
        \hline
        \hline
        Sheet 1 & 13, 35 \tabularnewline
        \hline
        Sheet 2 & 24, 46  \tabularnewline
        \hline
        Sheet 3 & 36 \tabularnewline
        \hline
        Sheet 4 & 15 \tabularnewline
        \hline
        Sheet 5 & 26 \tabularnewline
        \hline
        Sheet 6 & 14 \tabularnewline
        \hline
        Sheet 7 & 25 \tabularnewline
        \hline
\end{tabular}
    \par\vspace{0pt}
  \end{minipage}
  \begin{minipage}[b]{0.60\linewidth}
    \begin{tabular}{|l|l|l|}
        \hline
         5 Hopf links & 4 trefoils & 1 figure-eight \tabularnewline
        \hline
        (124)(356) & (136245)L & (142635)\tabularnewline
        (134)(256) & (136524)L & \tabularnewline
        (135)(246) & (143625)L & \tabularnewline
        (136)(245) & (146235)L & \tabularnewline
        (145)(236) & & \tabularnewline
        \hline
        \end{tabular}
    \par\vspace{0pt}
  \end{minipage}
\vspace{1 em}

  \noindent $K_6:7s5$  
\vspace{1 em}

\begin{minipage}[b]{0.35\linewidth}
\begin{tabular}{|l|l|}
        \hline
        Sheet Number & Edges\tabularnewline
        \hline
        \hline
        Sheet 1 & 13, 35 \tabularnewline
        \hline
        Sheet 2 & 24 \tabularnewline
        \hline
        Sheet 3 & 36 \tabularnewline
        \hline
        Sheet 4 & 15 \tabularnewline
        \hline
        Sheet 5 & 26, 46 \tabularnewline
        \hline
        Sheet 6 & 14 \tabularnewline
        \hline
        Sheet 7 & 25 \tabularnewline
        \hline
\end{tabular}
    \par\vspace{0pt}
  \end{minipage}
  \begin{minipage}[b]{0.60\linewidth}
    \begin{tabular}{|l|l|l|}
        \hline
        5 Hopf links & 4 trefoils & 2 figure-eights \tabularnewline
        \hline
        (124)(356) & (124635)R & (136425) \tabularnewline
        (125)(346) & (136245)L & (142635) \tabularnewline
        (134)(256) & (136524)L & \tabularnewline
        (136)(245) & (143625)L & \tabularnewline
        (145)(236) & &\tabularnewline
        \hline
        \end{tabular}
    \par\vspace{0pt}
  \end{minipage}

  \subsection*{8 sheets}  

  \noindent $K_6:8s1$ 
\vspace{1 em}

\begin{minipage}[b]{0.35\linewidth}
\begin{tabular}{|l|l|}
        \hline
        Sheet Number & Edges\tabularnewline
        \hline
        \hline
        Sheet 1 & 13, 36 \tabularnewline
        \hline
        Sheet 2 & 24 \tabularnewline
        \hline
        Sheet 3 & 35 \tabularnewline
        \hline
        Sheet 4 & 14 \tabularnewline
        \hline
        Sheet 5 & 26 \tabularnewline
        \hline
        Sheet 6 & 15 \tabularnewline
        \hline
        Sheet 7 & 46 \tabularnewline
        \hline
        Sheet 8 & 25 \tabularnewline
        \hline
\end{tabular}
    \par\vspace{0pt}
  \end{minipage}
  \begin{minipage}[b]{0.60\linewidth}
    \begin{tabular}{|l|l|l|}
        \hline
        3 Hopf links & 8 trefoils & 1 figure-eight \tabularnewline
        \hline
        (124)(356) & (134625)R & (142635) \tabularnewline
        (125)(346) & (135246)L & \tabularnewline
        (145)(236) & (135624)L & \tabularnewline
        & (142536)L & \tabularnewline
        & (146235)R & \tabularnewline \cline{2-2}
        & (13524)L & \tabularnewline
        & (14625)R & \tabularnewline
        & (24635)L & \tabularnewline
        \hline
        \end{tabular}
    \par\vspace{0pt}
  \end{minipage}
\vspace{1 em}

  \noindent $K_6:8s2$ 
\vspace{1 em}

\begin{minipage}[b]{0.35\linewidth}
\begin{tabular}{|l|l|}
        \hline
        Sheet Number & Edges\tabularnewline
        \hline
        \hline
        Sheet 1 & 13 \tabularnewline
        \hline
        Sheet 2 & 24 \tabularnewline
        \hline
        Sheet 3 & 35, 36 \tabularnewline
        \hline
        Sheet 4 & 46 \tabularnewline
        \hline
        Sheet 5 & 15 \tabularnewline
        \hline
        Sheet 6 & 26 \tabularnewline
        \hline
        Sheet 7 & 14 \tabularnewline
        \hline
        Sheet 8 & 25 \tabularnewline
        \hline
\end{tabular}
    \par\vspace{0pt}
  \end{minipage}
  \begin{minipage}[b]{0.60\linewidth}
    \begin{tabular}{|l|l|l|}
        \hline
        3 Hopf links & 1 Solomon & 7 trefoils \tabularnewline
        \hline
        (134)(256) & (135)(246) & (135246)L \tabularnewline
        (136)(245) & & (136245)L \tabularnewline
        (145)(236) & & (136524)L \tabularnewline
        & & (143625)L \tabularnewline
        & & (146235)L \tabularnewline \cline{3-3}
        & & (13524)L \tabularnewline
        & & (14625)L \tabularnewline
        \hline
        \end{tabular}
    \par\vspace{0pt}
  \end{minipage}
\vspace{1 em}

  \noindent $K_6:8s3$ 
\vspace{1 em}

\begin{minipage}[b]{0.35\linewidth}
\begin{tabular}{|l|l|}
        \hline
        Sheet Number & Edges\tabularnewline
        \hline
        \hline
        Sheet 1 & 13, 36 \tabularnewline
        \hline
        Sheet 2 & 24 \tabularnewline
        \hline
        Sheet 3 & 35 \tabularnewline
        \hline
        Sheet 4 & 46 \tabularnewline
        \hline
        Sheet 5 & 15 \tabularnewline
        \hline
        Sheet 6 & 26 \tabularnewline
        \hline
        Sheet 7 & 14\tabularnewline
        \hline
         Sheet 8 & 25 \tabularnewline
        \hline
\end{tabular}
    \par\vspace{0pt}
  \end{minipage}
  \begin{minipage}[b]{0.60\linewidth}
    \begin{tabular}{|l|l|l|}
        \hline
        3 Hopf links & 1 Solomon & 9 trefoils \tabularnewline
        \hline
        (124)(356) & (135)(246) & (124635)L \tabularnewline
        (134)(256) & & (135246)L \tabularnewline
        (145)(236) & & (142536)L \tabularnewline
        & & (142635)R \tabularnewline
        & & (143625)L \tabularnewline
        & & (146235)L \tabularnewline \cline{3-3}
        & & (12536)L \tabularnewline
        & & (13524)L \tabularnewline
        & & (14625)L \tabularnewline
        \hline
        \end{tabular}
    \par\vspace{0pt}
  \end{minipage}
\vspace{1 em}

\noindent $K_6:8s4$ 
\vspace{1 em}

\begin{minipage}[b]{0.35\linewidth}
\begin{tabular}{|l|l|}
        \hline
        Sheet Number & Edges\tabularnewline
        \hline
        \hline
        Sheet 1 & 13 \tabularnewline
        \hline
        Sheet 2 & 24, 26 \tabularnewline
        \hline
        Sheet 3 & 35 \tabularnewline
        \hline
        Sheet 4 & 46 \tabularnewline
        \hline
        Sheet 5 & 15 \tabularnewline
        \hline
        Sheet 6 & 36 \tabularnewline
        \hline
        Sheet 7 & 14\tabularnewline
        \hline
         Sheet 8 & 25 \tabularnewline
        \hline
\end{tabular}
    \par\vspace{0pt}
  \end{minipage}
  \begin{minipage}[b]{0.60\linewidth}
    \begin{tabular}{|l|l|}
        \hline
        5 Hopf links & 8 trefoils \tabularnewline
        \hline
        (125)(346) & (124635)L  \tabularnewline
        (134)(256) & (135246)L \tabularnewline
        (135)(246) & (136425)R \tabularnewline
        (136)(245) & (136524)L \tabularnewline
        (145)(236) & (146325)L \tabularnewline \cline{2-2}
        &  (13524)L \tabularnewline
        &  (13625)R \tabularnewline
        &  (14635)L \tabularnewline
        \hline
        \end{tabular}
    \par\vspace{0pt}
  \end{minipage}
\vspace{1 em}

\noindent $K_6:8s5$ 
\vspace{1 em}

\begin{minipage}[b]{0.35\linewidth}
\begin{tabular}{|l|l|}
        \hline
        Sheet Number & Edges\tabularnewline
        \hline
        \hline
        Sheet 1 & 13 \tabularnewline
        \hline
        Sheet 2 & 24 \tabularnewline
        \hline
        Sheet 3 & 35 \tabularnewline
        \hline
        Sheet 4 & 14, 46 \tabularnewline
        \hline
        Sheet 5 & 26 \tabularnewline
        \hline
        Sheet 6 & 15 \tabularnewline
        \hline
        Sheet 7 & 36\tabularnewline
        \hline
         Sheet 8 & 25 \tabularnewline
        \hline
\end{tabular}
    \par\vspace{0pt}
  \end{minipage}
  \begin{minipage}[b]{0.60\linewidth}
    \begin{tabular}{|l|l|l|}
        \hline
        5 Hopf links & 7 trefoils & 1 figure-eight \tabularnewline
        \hline
        (124)(356) & (124635)L & (142635)  \tabularnewline
        (125)(346) & (135246)L & \tabularnewline
        (135)(246) & (135624)R & \tabularnewline
        (136)(245) & (136425)R & \tabularnewline
        (145)(236) & (143625)R & \tabularnewline \cline{2-2}
        &  (13524)L & \tabularnewline
        &  (13625)R & \tabularnewline
        \hline
        \end{tabular}
    \par\vspace{0pt}
  \end{minipage}
\vspace{1 em}

  \noindent $K_6:8s6$ 
\vspace{1 em}

\begin{minipage}[b]{0.35\linewidth}
\begin{tabular}{|l|l|}
        \hline
        Sheet Number & Edges\tabularnewline
        \hline
        \hline
        Sheet 1 & 13 \tabularnewline
        \hline
        Sheet 2 & 24, 26 \tabularnewline
        \hline
        Sheet 3 & 35 \tabularnewline
        \hline
        Sheet 4 & 14 \tabularnewline
        \hline
        Sheet 5 & 36 \tabularnewline
        \hline
        Sheet 6 & 15 \tabularnewline
        \hline
        Sheet 7 & 46\tabularnewline
        \hline
         Sheet 8 & 25 \tabularnewline
        \hline
\end{tabular}
    \par\vspace{0pt}
  \end{minipage}
  \begin{minipage}[b]{0.60\linewidth}
    \begin{tabular}{|l|l|l|}
        \hline
        5 Hopf links & 7 trefoils & 1 figure-eight \tabularnewline
        \hline
        (124)(356) & (134625)R & (136425)  \tabularnewline
        (125)(346) & (135246)L & \tabularnewline
        (134)(256) & (136254)R &\tabularnewline
        (136)(245) & (142635)R &\tabularnewline
        (145)(236) & (146325)R &\tabularnewline \cline{2-2}
        &  (13524)L &\tabularnewline
        &  (14635)R &\tabularnewline
        \hline
        \end{tabular}
    \par\vspace{0pt}
  \end{minipage}

\subsection*{9 sheets}  

\noindent $K_6:9s1$ 
\vspace{1 em}

\begin{minipage}[b]{0.35\linewidth}
\begin{tabular}{|l|l|}
    \hline
    Sheet Number &  Edges \tabularnewline
    \hline
    Sheet 1 & 13   \tabularnewline
    \hline
    Sheet 2 & 24   \tabularnewline
    \hline
    Sheet 3 & 35  \tabularnewline
    \hline
    Sheet 4 & 46   \tabularnewline
    \hline
    Sheet 5 & 15  \tabularnewline
    \hline
    Sheet 6 & 26 \tabularnewline
    \hline
    Sheet 7 & 14  \tabularnewline
    \hline
    Sheet 8 & 36   \tabularnewline
    \hline
    Sheet 9 & 25   \tabularnewline
    \hline
\end{tabular}
    \par\vspace{0pt}
  \end{minipage}
  \begin{minipage}[b]{0.60\linewidth}
    \begin{tabular}{|l|l|l|}
        \hline
        5 Hopf links & 1 Solomon & 8 trefoils \tabularnewline
        \hline
        (124)(356) & (135)(246) & (124635)L   \tabularnewline
        (125)(346) & & (135246)L \tabularnewline
        (134)(256) & & (136254)R \tabularnewline
        (136)(245) & & (136425)R \tabularnewline
        (145)(236) & & (142635)R \tabularnewline
        & &  (146235)L \tabularnewline \cline{3-3}
        & & (13524)L \tabularnewline
        & & (14625)L \tabularnewline
        \hline
        \end{tabular}
    \par\vspace{0pt}
  \end{minipage}
\vspace{1 em}

\noindent $K_6:9s2$  
\vspace{1 em}

\begin{minipage}[b]{0.35\linewidth}
\begin{tabular}{|l|l|}
    \hline
    Sheet Number &  Edges \tabularnewline
    \hline
    Sheet 1 & 13  \tabularnewline
    \hline
    Sheet 2 & 24  \tabularnewline
    \hline
    Sheet 3 & 36 \tabularnewline
    \hline
    Sheet 4 & 15  \tabularnewline
    \hline
    Sheet 5 & 26 \tabularnewline
    \hline
    Sheet 6 & 14 \tabularnewline
    \hline
    Sheet 7 & 35  \tabularnewline
    \hline
    Sheet 8 & 46  \tabularnewline
    \hline
    Sheet 9 & 25  \tabularnewline
    \hline
\end{tabular}
    \par\vspace{0pt}
  \end{minipage}
  \begin{minipage}[b]{0.60\linewidth}
    \begin{tabular}{|l|l|l|}
        \hline
        7 Hopf links & 6 trefoils & 3 figure-eights \tabularnewline
        \hline
        (124)(356)& (125364)L & (135264) \tabularnewline
        (125)(346)& (135246)L & (136425) \tabularnewline
        (134)(256)& (136245)L & (142635) \tabularnewline
        (135)(246)& (136524)L &  \tabularnewline
        (136)(245)& (143625)L &\tabularnewline
        (145)(236)& (146235)L &\tabularnewline
        (146)(235)& &\tabularnewline
        \hline
        \end{tabular}
    \par\vspace{0pt}
  \end{minipage}

\section{Acknowledgements}
The author would like to acknowledge Merrimack College graduates Amanda DeMarco, Stephen Francis, Joseph Mantoni, Andrew O'Hara who studied book representations during spring 2013 as a part of the course MTH4850: Directed Research in Mathematics at Merrimack College.  Their work contributed to identifying 3-sheet, 4-sheet and 9-sheet book representations of $K_6$.


\end{document}